\documentclass[10pt,a4paper]{amsart}
\usepackage{amssymb,amscd}
\usepackage[all,cmtip]{xy}
\usepackage{hyperref}

\newtheorem{thm}{Theorem}[section]
\newtheorem{lem}[thm]{Lemma}
\newtheorem{prop}[thm]{Proposition}

\theoremstyle{definition}
\newtheorem{defn}[thm]{Definition}
\newtheorem{rem}[thm]{Remark}

\begin{document}

\title[The Dold-Kan correspondence and coalgebra structures]{The Dold-Kan correspondence and coalgebra structures}
\author{W. Hermann B. Sore }
\address{Departement de Mathematiques et Informatique, UFR/ST, Universite Polytechnique de Bobo-Dioulasso, 01 BP 1091 Bobo-Dioulasso 01, Burkina Faso.}
\email{hermann.sore@gmail.com}
\date{\today}
\thanks{
 The author is indebted to Birgit Richter for useful suggestions and discussions on the subject.}


\begin{abstract}
By using the Dold-Kan correspondence we construct a Quillen adjunction between the model categories of \textit{non}-cocommutative coassociative simplicial and differential graded coalgebras over a field. We restrict to categories of connected coalgebras and prove a Quillen equivalence between them.
\end{abstract}

\maketitle

\section{Introduction}
In \cite{SS00}, Schwede and Shipley give sufficient conditions for lifting the model category structures of closed monoidal model categories to their categories of monoids along the free-forgetful adjoint pairs.
Then in \cite{SS03}, they consider a Quillen equivalence between the closed monoidal model categories and state conditions for obtaining a Quillen equivalence at the level of categories of monoids
as illustrated in the following diagram  
$$
\xymatrix{
 (\mbox{\textbf{C}}, \widehat\otimes, I_{\textbf C}) \ar@<+0.5ex>[d]^{\textrm{Free}} \ar@<+0.5ex>[rr]^{\simeq_{\mathrm{Quillen}}}       &       &   (\mbox{\textbf{D}}, \otimes, I_{\textbf D})  \ar@<+0.5ex>[d]^{\textrm{Free}} \ar@<+0.5ex>[ll]   \\
\mbox{\textbf{C.\textrm{Monoids}}} \ar@<+0.5ex>[u]^{\textrm{Forget}}  \ar@{.>}[rr]<+0.5ex>      &       &  \mbox{\textbf{D.\textrm{Monoids.}}}   \ar@{.>}[ll]<+0.5ex> \ar@<+0.5ex>[u]^{\textrm{Forget}}
}$$ 
In this paper we restrict to the categories of simplicial and differential non-negatively graded vector spaces. Both categories have closed monoidal model structures and are equivalent via the classical Dold-Kan correspondence. We consider their respective categories of comonoids, namely, simplicial and differential non-negatively graded coalgebras and we try to mimic dual methods from \cite{SS03} for providing the categories of comonoids with a Quillen equivalence. We prove that Dold-Kan's normalization functor descends to the level of categories of coalgebras and gives rise to an adjoint pair. Moreover, we produce model categories structures for simplicial and differential non-negatively graded coalgebras (we refer to \cite{Goe95} and to the unpublished \cite{GG99} where these model category structures have been studied at first) and observe that the adjunction obtained turns out to be a Quillen pair. 
We then restrict to the categories of simplicial and differential non-negatively graded \textit{connected} coalgebras. We derive corresponding model structures for these categories and establish a Quillen equivalence between them.

Section 2 is devoted to some generalities about the monoidal categories of differential non-negatively graded vector spaces, simplicial vector spaces and their corresponding categories of comonoids.\\
Section 3 revisits the model structure on the category of \textit{non}-cocommutative coassociative differential graded coalgebras over a field. This model category structure is due to Getzler and Goerss in their unpublished paper \cite{GG99}. We reproduce with more details most of their results and proofs. However we slightly modify their arguments for the cofibration-(acyclic fibration) axiom.\\
Section 4 shows that the normalization functor $N$ defines a functor from the category of simplicial coalgebras to the category of differential non-negatively graded coalgebras. Similar observations hold for the inverse functor $\Gamma$, however, as in \cite{SS03}, we point out that both functors are not adjoint on the level of comonoids. We then construct a functor that is right adjoint to the normalization functor $N$ on the level of comonoids. The adjoint pair of functors obtained in this way turns out to be an adjoint Quillen pair. We refer to \cite{Ric03} where similar constructions are also considered.\\
Section 5 restricts to the categories of connected simplicial coalgebras and connected differential graded coalgebras. In this section, we first investigate the completeness and cocompleteness properties of connected differential graded coalgebras. We then lift along a suitable adjunction the model structure on the category of differential graded coalgebras to the category of connected differential graded coalgebras by means of the transfer principle. In order to obtain the (acyclic cofibration)-fibration factorization axiom in a functorial way we adapt and reprove some useful techniques from \cite{Smi} in our setting. We end this section by proving a Quillen equivalence between the categories of connected differential graded and simplicial coalgebras.\\
Section 6 closes this paper with an appendix on the category of connected differential graded algebras. The results therein are intended as a step for constructing limits for connected differential graded coalgebras.

\section{Preliminaries}
\subsection{Differential graded vector spaces}
In this section we review some basic results from homological algebra. For a thorough treatment, we refer to \cite{Mac63}. 
Let \textbf{Vct} be the category of vector spa\-ces over a fixed field $K$. We denote by \textbf{DGVct} the category of differential gra\-ded $K$-vectors spaces which are concentrated in non-negative degrees and have differentials of degree $-1$. We denote by $\mathbb{S}^n$ the $n$-sphere chain complex. This is the object of \textbf{DGVct} which has the field $K$ in degree $n$ and $0$ in other degrees. All differentials in $\mathbb{S}^n$ are trivial. The $n$-disk,  denoted by $\mathbb{D}^n$, is the object of \textbf{DGVct} which has the field $K$ in degrees $n$ and $n-1$ and $0$ elsewhere. The identity on $K$ is its only non-trivial differential and we set $\mathbb D^0 = \mathbb S^0$.

Recall that if $X$ and $Y$ are two objects in the category \textbf{DGVct},
a symmetric monoidal product $\otimes$ is given by $(X\otimes Y)_n = \bigoplus_{p+q=n} X_p\otimes_K Y_q$
with differential
$d(x\otimes y) = dx\otimes y + (-1)^{\mid x\mid }x\otimes dy$.
The unit of this monoidal product is $\mathbb{S}^0$, the
differential graded vector space concentred in degree $0$, which we sometimes denote by $K[0]$.
The category \textbf{DGVct} endowed with the monoidal product $\otimes$ is closed.
Given differential graded vector spaces $(X,d_X)$ and $(Y,d_Y)$, let $\mathrm{Hom}(X,Y)$ be the object of \textbf{DGVct} with:

\begin{eqnarray*}
\mathrm{Hom}(X,Y)_0 &=& \left\{f \in \prod_{p \geq 0} \textbf{Vct}(X_p, Y_p) \mid d_Y(f_px)-f_{p-1}(d_Xx)=0, \hspace{.1cm} x \in X_p\right\} \\
\mathrm{Hom}(X,Y)_n &=& \prod_{p \geq 0} \textbf{Vct}(X_p,Y_{p+n}), \hspace{.2cm} \mbox{for} \hspace{.2cm} n\geq 1 
\end{eqnarray*}

and differential $d_H$ for any map $f = \{f_p \colon X_p \rightarrow Y_{p+n} \}_{p \geq 0}$ given by
$$(d_Hf)_p(x) = d_Y (f_p x) + (-1)^{n+1}f_{p-1}(d_X x), \quad  x \in X_p. $$
The specified right adjoint of the functor $-\otimes Y \colon \mbox{\textbf{DGVct}} \rightarrow \mbox{\textbf{DGVct}}$ is then given by the functor $\mathrm{Hom}(Y,-) \colon \mbox{\textbf{DGVct}} \rightarrow \mbox{\textbf{DGVct}}$. This right adjoint functor is given by applying the good truncation below $0$ (see \cite[1.2.7]{Wei94}) to the unbounded version of $\mathrm{Hom}(X,Y)$.

\subsection{Simplicial vector spaces}
We denote by \textbf{SVct} the category of simplicial vector spaces, that is, the category of functors
$X \colon \Delta^{\mathrm{op}} \rightarrow \textbf{Vct}$, where $\Delta$ is the category of finite ordered sets $\left[n\right] = \left\{0<1<\ldots<n\right\}$ for $n \in \mathbb N$ and whose morphisms are non-decreasing monotone functions (see for instance \cite[8.1]{Wei94}).

Let $X$ and $Y$ be two objects in \textbf{SVct}. A monoidal product $\widehat \otimes$ is given by
$(X \widehat \otimes Y)_n = X_n \otimes_K Y_n$ with coordinatewise structure maps. The unit of the monoidal product $\widehat \otimes $ is the simplicial vector space having $K$ in each degree and identity maps on $K$ as face and degeneracy operators. We denote this unit by $I(K)$. This monoidal product $\widehat \otimes $ is symmetric and  \textbf{SVct} is closed.

\subsection{Differential graded coalgebras}
 We denote by \textbf{DGcoAlg} the category of counital coassociative differential graded $K$-coalgebras. In other words \textbf{DGcoAlg} is the category of comonoids in the monoidal category (\textbf{DGVct}, $\otimes$, $K[0]$).

\subsection{Simplicial coalgebras}
The category of simplicial coalgebras, denoted by \textbf{ScoAlg}, is the category of comonoids
in the monoidal category (\textbf{SVct}, $\widehat\otimes$, $I(K)$).

\vspace{.2cm}
In the coming sections the symbols $\sqcap$ and $\sqcup$ will stand respectively for the categorical product and coproduct in the appropriate categories.

\section{Model category structures on categories of coalgebras} \label{modcatdgs}

\subsection{A model category structure on \textbf{DGcoAlg}} \label{subsection:modelcoalg}

In this section, we consider the category of \textit{non}-cocommutative coassociative, counital differential non-negatively graded coalgebras over a fixed field $K$, denoted here by \textbf{DGcoAlg}. 
We review the different arguments for proving the following result due to Getzler and Goerss in their unpublished paper \cite{GG99}: 

\begin{thm} \cite[Definition 2.3, Theorem 2.8]{GG99}.  \label{defn:model} \\
Define $f \colon C \rightarrow D \in $ \textbf{DGcoAlg} to be  
\begin{enumerate}
\item[1.] a \emph{weak equivalence} if $H_\ast f$ is an isomorphism.
\item[2.] a \emph{cofibration} if $f$ is a degreewise injection of graded vector spaces.
\item[3.] a \emph{fibration} if $f$ has the right lifting property with respect to acyclic cofibrations.
\end{enumerate}
With these definitions, \textbf{DGcoAlg} becomes a closed model category.
\end{thm}

Before proving the theorem stated above, we need to establish an analogue of the Fundamental Theorem on Coalgebras (see \cite[Theorem 2.2.1]{Swe69}) for the objects of \textbf{DGcoAlg}. Namely if $C$ is an object of \textbf{DGcoAlg} and $c \in C$ is an homogeneous element then the sub-coalgebra generated by $c$ is finite dimensional. We recall that given a graded coalgebra $C$, a graded right $C$-comodule $M$ is a graded vector space with a structure map $\omega \colon M\rightarrow M\otimes C$. For a homogeneous basis $\left\{c_i\right\}$ of $C$ the structure map is given by $\omega(m)=\sum_i m_i \otimes c_i$ under Sweedler's notation with all but finitely many of the $m_i=0$. Moreover the graded linear dual $C^\ast=\left\{\textbf{Vct}(C_n, K)\right\}$ becomes a non-positively graded algebra and $M$ inherits a left $C^\ast$-module structure by setting $f.m = \sum_i \left\langle f,c_i\right\rangle m_i$ for $f\in C^\ast$. These facts are well-known and may be found in \cite[Chapter 2]{Swe69} for the ungraded case or in \cite[Section 2.5]{Hov99} for comodule over differential graded Hopf algebras.

\begin{lem} \cite[Lemma 1.1]{GG99}. \label{comod-lem} \\
Let $\left(C, \Delta_C, \epsilon_C\right)$ be a graded coalgebra and 
$\left(M, \omega \colon M \rightarrow M \otimes C\right)$ be a right graded $C$-comodule. If $x\in M$ is a homogeneous element, then the subcomodule generated by $x$ is finite-dimensional.
\end{lem}
\begin{proof}
Let $\left\{c_i\right\}$ be a homogeneous basis for $C$. Then $\omega (x)=\sum_i x_i\otimes c_i$ with all but finitely many $x_i$ being zero. Now let $N$ denote the vector space spanned by the $x_i$. The identity $\mathrm{id}_M = (\mathrm{id}_M\otimes \epsilon_C)\circ \omega$ yields $x=\sum_i x_i.\epsilon_C(c_i)$ and hence we conclude that $x\in N$. Since $N$ is finite-dimensional it remains to check that $N$ is a subcomodule of $M$, that is, $\omega(N)\subseteq N\otimes C$. This is obtained by performing the following computations
\begin{eqnarray*}
\sum_i \omega(x_i) \otimes c_i &= &(\omega \otimes \mathrm{id}_C)\circ \omega(x) \\
                          & = & (\mathrm{id}_M \otimes \Delta_C)\circ \omega(x)\\
													&=& (\mathrm{id}_M \otimes \Delta_C)\left( \sum_j x_j\otimes c_j \right)\\
													&=& \sum_j x_j \otimes \Delta_C(c_j) \\
													&=& \sum_{j,i} x_j \otimes c_{ij} \otimes c_i 
\end{eqnarray*}
for some $c_{ij} \in C$. This yields $\omega(x_i)=\sum_j x_j\otimes c_{ij} \in N\otimes C$.
\end{proof}

\begin{lem} \cite[Lemma 1.2]{GG99}. \label{gradcoalg-lem} \\
Let $C$ be a graded coalgebra and $x\in C$ a homogeneous element. Then there is a finite dimensional sub-coalgebra $D\subseteq C$ such that $x\in D$. Moreover it can be assumed that $D_n=0$ for $n$ larger that the degree of $x$.
\end{lem}
\begin{proof}
 Since the coalgebra $C$ is a right $C$-comodule with structure map given by $\omega=\Delta_C$, the Lemma \ref{comod-lem} above supplies a finite-dimensional subcomodule with $x\in N \subseteq C$. Note that there is a comodule structure map $\omega \colon N \rightarrow N\otimes C$ defined by $\omega(n) = \sum_i n_i\otimes c_i$ for any homogeneous $n\in N$.
This comodule structure induces a left $C^\ast$-module structure on $N$ by setting $f.n=\sum_i\left\langle f,c_i\right\rangle n_i$ for all $f\in C^\ast$ and homogeneous $n\in N$.
Now let $\varphi \colon C^\ast \rightarrow \mathrm{End}_K(N)$ be the morphism induced by the above left $C^\ast$-module structure on $N$ and consider the orthogonal 
$$\left(\ker \varphi\right)^\bot =\left\{y \in C \mid \left\langle f,y\right\rangle =0 \hspace{.2cm} \forall f\in \ker \varphi\right\}.$$ 
Firstly $\left(\ker \varphi\right)^\bot$ is a sub-coalgebra of $C$ since $\ker \varphi$ is an ideal of $C^\ast$. Secondly 
$$ \dim_K\left(\ker \varphi\right)^\bot = \dim_K \left(C^\ast / \left(\ker \varphi \right)^{\bot \bot}\right) \leq \dim_K \left(C^\ast / \ker \varphi\right) .$$ 
But $\dim_K \left(C^\ast / \ker \varphi\right)$ is finite by the First Isomorphism Theorem on $\varphi$ and the fact that $N$ is finite dimensional. Hence we conclude that $\left(\ker \varphi\right)^\bot$ is finite dimensional.
Since $x\in N \subseteq \left(\ker \varphi\right)^\bot$ we may set $D= \left(\ker \varphi\right)^\bot$. Finally in order to have $D_n=0$ for $n$ greater than the degree of $x$, we just consider the sub-coalgebra of $D$ generated by the homogeneous elements of degree less than or equal to that of $x$.
\end{proof}

\begin{prop} \cite[Proposition 1.5]{GG99}. \label{FundTheo-DGcoalg} \\
Let $(C, \partial)$ be a differential graded coalgebra and $x\in C$ a homogeneous element. Then there is a finite dimensional differential graded coalgebra $D\subseteq C$ such that $x\in D$.
\end{prop}
\begin{proof}
Suppose $x \in C$ is a homogeneous element degree of degree $n$. By Lemma \ref{gradcoalg-lem} there is a finite dimensional graded sub-coalgebra denoted $D(n)$ such that $x\in D(n) \subseteq C$ and $D(n)_k =0$ for $k>n$. Then choose a basis $\left\{y_i\right\}$ for $D(n)_n$ and use again Lemma \ref{gradcoalg-lem} to produce finite dimensional sub-coalgebras $D(y_i) \subseteq C$ for each $y_i$ such that $\partial_n y_i \in D(y_i)_{n-1}$.
Then set $$D(n-1)=D(n)+ \sum_i D(y_i).$$
Thus $D(n-1)$ is a sub-coalgebra since the sum of coalgebras is again a coalgebra. Moreover $D(n-1)$ is finite dimensional. The process for obtaining $D(n-1)$ may be repeated to form an ascending sequence of sub-coalgebras
$$D(n)\subseteq D(n-1) \subseteq \cdots \subseteq D(0) \subseteq C$$
with the properties that for $0\leq k\leq n-1$
\begin{enumerate}
\item each $D(k)$ is finite dimensional,
\item $D(k-1)_l = D(k)_l$ for $l\geq k$,
\item $\partial_k \left(D(k)_k\right) \subseteq D(k-1)_{k-1}$.
\end{enumerate}
Finally setting $D=D(0)$ gives the required result.
\end{proof}

With this key result in hand we can now proceed to the proof of the model category axioms. We refer to \cite[Definition 3.3]{DS95} for the numbering of model category axioms. The coming sections will essentially focus on the (co)completeness and factorizations axioms.

\subsubsection{\textbf{Axiom MC1}} \label{axiom MC5(i)}

This section aims at proving the completeness of the category of coalgebras. The category \textbf{DGcoAlg} turns out to be anti-equivalent to a category in which colimits are easier to describe. The completeness of \textbf{DGcoAlg} is then derived from the cocompleteness of that category. \\

We denote by \textbf{ProDGAlg$_{\leq 0}$} the category of profinite non-positively graded differential algebras. An object $A$ in this category is an inverse limit $\lim_\alpha A_\alpha$ of degreewise finite dimensional graded differential algebras $A_\alpha$. By endowing the finite dimensional algebras $A_\alpha$ with the discrete topology, the object $A \in$ \textbf{ProDGAlg$_{\leq 0}$} inherits a topology where two-sided ideals of finite codimension form a neighborhood basis of $0$. Moreover, for an object $A \in$ \textbf{ProDGAlg$_{\leq 0}$}, the continuous linear dual $A'$ defined degreewise by 
$$A_n' = \left\{f \colon A_{-n} \longrightarrow K \mid \ker f \hspace{.1cm}\mbox{is open}\right\}$$
becomes an object of \textbf{DGcoAlg}, since by \cite[Proposition 6, Page 290]{Koe67} the continuous linear dual takes limits to colimits.

\begin{lem}  \cite[Proposition 1.7]{GG99}. \label{lem:antiequiv} \\
The algebraic linear dual and the continuous linear dual define an anti-equivalence of categories between \textbf{DGcoAlg} and \textbf{ProDGAlg$_{\leq 0}$}.
\end{lem}

\begin{proof}
By Proposition \ref{FundTheo-DGcoalg} any $C \in $ \textbf{DGcoAlg} can be written as a filtered colimit of its finite dimensional subcoalgebras.
It follows that $$\left(C^\ast\right)' = \left((\mathrm{colim} C_\alpha)^\ast\right)' \cong (\lim C_\alpha^\ast)' \cong \mathrm{colim} \left(C_\alpha^\ast\right)' \cong \mathrm{colim C_\alpha} = C.$$
For $A\in$ \textbf{ProDGAlg$_{\leq 0}$}, a similar argument yields $\left(A^\ast\right)' \cong A.$
\end{proof}

\begin{prop}  \cite[Proposition 1.8]{GG99}. \label{prop:dgcocomplete} \\
The category \textbf{DGcoAlg} is complete and cocomplete.
\end{prop}

\begin{proof}
The forgetful functor from \textbf{DGcoAlg} to \textbf{DGVct} creates colimits and therefore the cocompleteness of \textbf{DGcoAlg} follows immediately.
Then, we note that \textbf{ProDGAlg$_{\leq 0}$} is cocomplete. Indeed, let $A \colon I \rightarrow$ \textbf{ProDGAlg$_{\leq 0}$} be a diagram of profinite differential graded algebras. Here are the steps for defining its colimit in \textbf{ProDGAlg$_{\leq 0}$}:
\begin{enumerate}
\item form the colimit $B = \mathrm{colim}_{i \in I} A_i$ of this diagram in the category of non-positively graded differential algebras. 
\item endow $B$ with the topology where a neighborhood basis of $0$ is given by the set $\mathcal{J}$ of two-sided ideals $J$ which can be realized as the kernel of maps of algebras $B\rightarrow C$ such that $C$ is a finite dimensional differential graded algebra and such that  for each $i \in I$, the composite $A_i \rightarrow B \rightarrow C$ is continuous.
\item define the required colimits as the profinite completion $\lim_{J \in \mathcal{J}} B/J$ of $B$ with respect to this topology.
\end{enumerate}
By Proposition \ref{lem:antiequiv}, the continuous linear dual carries colimits in \textbf{ProDGAlg$_{\leq 0}$} to limits in \textbf{DGcoAlg} and therefore \textbf{DGcoAlg} is complete.
\end{proof}

\subsubsection{\textbf{Axiom MC5(i)}} In this section, we prove the cofibration-(acyclic fibration) axiom. First, we construct a functor from \textbf{DGVct} to \textbf{DGcoAlg} that is right adjoint to the forgetful functor. Then, this functor is used to provide the required factorization axiom in a functorial way. The arguments used here differ from that of \cite[Lemma 1.12, Theorem 2.1]{GG99} and refer rather to that used in the proof of \cite[Corollary 4.15]{Smi}. 

\begin{lem} \cite[Lemma 1.9]{GG99}. \label{lem:small}  \\
Let $\left\{C_\alpha\right\}_\alpha$ be a right filtered diagram in \textbf{DGcoAlg} and $D$ be a finite dimensional 
object in \textbf{DGcoAlg},
then the natural map
$$\mathrm{colim}_\alpha \textbf{DGcoAlg}\left(D, C_\alpha\right) \longrightarrow \textbf{DGcoAlg}\left(D, \mathrm{colim}_\alpha C_\alpha\right)$$
is a bijection.
\end{lem}

\begin{prop}  \cite[Proposition 1.10]{GG99}. \label{prop:rightadjoint} \\
Let $V$ be an object in \textbf{DGVct}. Then, there is a functor denoted $S_d$ from \textbf{DGVct} to \textbf{DGcoAlg} that is right adjoint to the functor $U_d$ that forgets coalgebra structure:
\begin{enumerate}
\item if $V$ is degreewise finite dimensional, $$S_d(V) = \left(\widehat{T_d(V^\ast)}\right)'$$
\item for any $V \in$ \textbf{DGVct}, $$S_d(V) = \mathrm{colim}_\alpha (S_d(V_\alpha))$$ with $V_\alpha$ running over finite dimensional subvector spaces of $V$.
\end{enumerate}
\end{prop}

\begin{proof}
  If $V$ is finite-dimensional using the following bijections with any object $D \in $ \textbf{DGcoAlg}
\begin{eqnarray*}
  \textbf{DGcoAlg} (D, (\widehat{T_d(V^\ast)})')   &   \cong   & \textbf{ProDGcoAlg}_{\leq 0} ( \widehat{T_d(V^\ast)}, D^\ast)  \\
                                             &   \cong   &  \textbf{DGAlg}_{\leq 0} ( T_d(V^{\ast}), D^\ast) \\
                                             & \cong  &     \textbf{ProDGVct}_{\leq 0} ( V^\ast, D^\ast) \\
                                            &   \cong   &  \textbf{DGVct} ( \left( D^\ast\right)', \left(V^\ast\right)') \\
																						&  \cong &  \textbf{DGVct}(D, V)
  \end{eqnarray*}
  give the desired result.
  
Now consider a general $V \in$ \textbf{DGVct}. Any object $D \in$ \textbf{DGcoAlg} can be written as the colimit $D = \mathrm{co}\!\lim_{\beta} D_\beta$ of its finite-dimensional subcoalgebras. Then, by Lemma \ref{lem:small} the following bijections
\begin{eqnarray*}
  \textbf{DGcoAlg}(D,\mathrm{colim}_\alpha (S_d(V_\alpha))) &   \cong   &   \lim_{\beta} \textbf{DGcoAlg}(D_\beta,\mathrm{colim}_\alpha (S_d(V_\alpha)))       \\
                                  &   \cong   &   \lim_{\beta} \mathrm{co}\!\lim_{\alpha} \textbf{DGcoAlg} (D_\beta,S_d(V_\alpha))      \\
                                  &   \cong   &   \lim_{\beta} \mathrm{co}\!\lim_{\alpha} \textbf{DGVct}(D_\beta,V_\alpha)         \\ 
                                  &   \cong   &   \textbf{DGVct}(D,V)
 \end{eqnarray*}
complete the proof.
\end{proof}

\begin{defn}
Let $K$ be a field and $K\left\langle x\right\rangle$ denote the vector space with basis x.
We denote by $(I,d_I)$ the \textit{unit interval}, i.e., the following differential graded vector space concentrated in degrees $1$ and $0$:
$$ \cdots 0 \rightarrow 0 \rightarrow K\left\langle a \right\rangle \stackrel{d}{\longrightarrow} K\left\langle b,c\right\rangle \rightarrow 0 \cdots $$
with $d(a) = c-b$.
Setting $\Delta (b) = b \otimes b $, $\Delta (c) = c \otimes c$ and $\Delta (a) = b \otimes a + a \otimes c $
yields a coassociative counital coalgebra structure on the differential graded vector space $(I,d_I)$. 
\end{defn}

\begin{lem}
Let $f, g \colon V \rightarrow W$ be morphisms of \textbf{DGVct}. A \textit{chain homotopy} between $f$ and $g$ is a chain map 
$H \colon V \otimes I \rightarrow W$ such that $H(v \otimes b) =f(v)$ and $H(v \otimes c) =g(v)$.
\end{lem}

\begin{proof}
For $v \in V_n$, it suffices to set $s_n(v) = (-1)^n H_{n+1}(v\otimes a)$ to recover the classical chain homotopy definition.
\end{proof}

\begin{lem} \cite[Proposition 4.10]{Smi}. \label{lem:Shomotopy} \\
Let $f, g \colon V \rightarrow W$ be chain homotopic morphisms in \textbf{DGVct}. \\
Then, $S_d(f), S_d(g) \colon S_d(V) \rightarrow S_d(W)$ in \textbf{DGcoAlg} are chain homotopic in \textbf{DGcoAlg}. 
\end{lem}

\begin{proof}
There is a homotopy $H \colon V \otimes I \rightarrow W$. Applying the cofree functor to this homotopy yields a morphism of coalgebras $S_d(H) \colon S_d(V \otimes I) \rightarrow S_d(W)$. We note that $S_d(V) \otimes I$ inherits a coalgebra structure from that of $S_d(V)$ and $I$ with a comultiplication given by the composite 
$$\tau \circ \Delta_{S_d(V)} \otimes \Delta_{I} \colon S_d(V) \otimes I \rightarrow S_d(V) \otimes S_d(V) \otimes I \otimes I \rightarrow S_d(V) \otimes I \otimes S_d(V) \otimes I$$
where $\tau$ is the switch map.
The morphism $S_d(V) \otimes I \rightarrow V \otimes I$, together with the universal property of the cofree coalgebra functor $S_d$ yields a morphism $\pi$
$$
\xymatrix{
V \otimes I & S_d(V) \otimes I \ar[l] \ar@{-->}[ld]^\pi \\
S_d(V \otimes I) \ar[u]
}
$$
Finally, the composite $S_d(H) \circ \pi \colon S_d(V) \otimes I \rightarrow S_d(V \otimes I) \rightarrow S_d(W)$ gives the required homotopy. 
\end{proof}

\begin{lem} \cite[Proposition B.17]{Smi}. \label{lem:producthomotopy} \\
Let $f \colon A \rightarrow B$ and $g \colon C \rightarrow D$ be morphisms in \textbf{DGcoAlg}. If $f$ and $g$ are homotopy equivalences, then $f \sqcap g \colon A \sqcap C \rightarrow B \sqcap D$ is a homotopy equivalence.
\end{lem}

\begin{prop}  \cite[Lemma 2.4 2.)]{GG99}. \label{prop:mc5i} \\
Let $f \colon C \rightarrow D$ be a morphism in \textbf{DGcoAlg}. 
Choose an acyclic differential graded vector space $V$ containing $C$.
Then the morphism $f$ can be factored
$$ C \stackrel{i}{\longrightarrow} D \sqcap S_d(V) \stackrel{p}{\longrightarrow} D$$
with $i$ a cofibration and $p$ an acyclic fibration.
\end{prop}

\begin{proof}
First forget the coalgebra structure on 
$C$ by considering $U_d(C) \in$ \textbf{DGVct}. Then define $V$ to be cone($U_d(C)$). The object 
$\mathrm{cone}(U_d(C)) \in$ \textbf{DGVct} is acyclic and comes with a canonical embedding $j \colon U_d(C) \rightarrow \mathrm{cone}(U_d(C))$.
 
We define $p$ to be the projection map $D \sqcap S_d(V) \rightarrow D$.
The homotopy equivalence between $\mathrm{cone} U_d(C)$ and $0$ yields a homotopy equivalence between $S_d(\mathrm{cone} U_d(C))$ and $S_d(0) = K[0]$ by Lemma \ref{lem:Shomotopy}. Then, the morphism $p \colon D \sqcap S_d(\mathrm{cone} U_d(C)) \rightarrow D \sqcap K[0] \cong D$ is a homotopy equivalence by Lemma \ref{lem:producthomotopy}. Note that we are working with bounded chain complexes over a fixed field, and hence, homotopy equivalences correspond to quasi-isomorphisms. It follows that the morphism $p$ is acyclic.

Then, to give a  morphism in  \textbf{DGcoAlg}$\left(C, D\sqcap S_d(V)\right)$
amouts to give a pair of  morphisms in  $$\textbf{DGcoAlg}(C,D) \times \textbf{DGcoAlg}(C,S_d(V)).$$
But using the adjunction between the categories of coalgebras and vector spaces, the previous product is equivalent to $$\textbf{DGcoAlg}(C,D) \times \textbf{DGVct}(U_d(C),V).$$
In this way, the pair $(f,j)$ yields a coalgebra morphism $i$. Now consider $\pi \colon D \sqcap S_d(V) \rightarrow S_d(V)$ the projection onto $S_d(V)$ and the map $\epsilon \colon S_d(V)\rightarrow V$ coming from the counit of the adjunction between coalgebras and vector spaces.
Then the following composite
$$C \stackrel{i}{\longrightarrow} D \sqcap S_d(V) \stackrel{\pi}{\longrightarrow} S_d(V) \stackrel{\epsilon}{\longrightarrow} V =\mathrm{cone}(U_d(C))$$
is the embedding $j$ and therefore insures that the coalgebra morphism $i$ is a cofibration as required.
\end{proof}

\subsubsection {\textbf{Axiom MC5(ii)}}
In this section, we prove the (acyclic cofibration)-fibration axiom. We exhibit a class of morphisms in \textbf{DGcoAlg} and use the small object argument to perform the required factorization axiom with respect to this class. 

\begin{lem}  \cite[Lemma 2.5]{GG99}. \label{lem:homogeneous} \\
Let $j \colon C \rightarrow D$ be an acyclic cofibration in \textbf{DGcoAlg} and $x \in D$ be a homogeneous element. Then there exists a subcoalgebra $B$ of $D$ such that 
\begin{enumerate}
\item[1.] $x \in B$
\item[2.] $B$ has a countable homogeneous basis,
\item[3.] $C \cap B \rightarrow B$ is an acyclic cofibration in \textbf{DGcoAlg}.
\end{enumerate}
\end{lem}

\begin{proof}
The idea of the proof is to construct a sequence $\left(B(n)\right)_{n\geq 1}$ of subcoalgebras of $D$ having the following properties:
\begin{enumerate}
\item  $B(1) \subseteq B(2) \subseteq \cdots \subseteq B(n) \cdots$,
\item each $B(n) \in$ \textbf{DGcoAlg} is finite dimensional,
\item the induced map $B(n-1)/\left[C \cap B(n-1)\right] \rightarrow B(n)/\left[C \cap B(n)\right]$ of differential graded vector spaces is zero in homology.
\end{enumerate}
First, there is a finite dimensional subcoalgebra $B(1)$ of $D$ that contains $x$.
Then, suppose that $B(n-1)$ has been constructed. Since $B(n-1)$ is finite dimensional, choose a finite set of homogeneous cycles $z_i + C \cap B(n-1) \in B(n-1)/\left[C\cap B(n-1)\right]$ so that the resulting homology classes span $H_\ast\Big(B(n-1)/\left[C\cap B(n-1)\right]\Big)$.
For each index $i$, there is a homogeneous element $x_i \in D$ which is a boundary of $z_i$ since $H_\ast\left(D/C\right)$ is zero. Then choose a finite dimensional subcoalgebra $x_i \in A(x_i) \subseteq D$ and set $B(n) = B(n-1) + \sum_i A(x_i)$.
Finally, setting $B = \bigcup_n B(n)$ satisfies the statements of the lemma. Indeed, we have
$$H_\ast (C \cap B) = H_\ast (\bigcup_n C \cap B(n) ) = \bigcup_n H_\ast(C \cap B(n)) = H_\ast C.$$
This comes from the facts that homology commutes with filtered colimits and that $H_\ast(C \cap B(n)) \cong H_\ast (B(n))$ due to the long exact sequence on homology resulting from the exact sequence $0 \rightarrow C \cap B(n) \rightarrow B(n) \rightarrow B(n)/\left[C \cap B(n)\right] \rightarrow 0.$ 
\end{proof}

\begin{lem} \cite[Lemma 2.6]{GG99}. \label{gen-triv-cof} \\
A morphism $q \colon X \rightarrow Y$ in \textbf{DGcoAlg} is a fibration if and only if it has the right lifting property with respect to all acyclic cofibrations $A \rightarrow B$ so that $B$ has a countable homogeneous basis.
\end{lem}

\begin{proof}
The first implication follows immediately from the definition of fibration given in Definition \ref{defn:model}. For the second implication, suppose that a morphism $q \colon X \rightarrow Y \in$ \textbf{DGcoAlg} has the right lifting property with respect to all acyclic cofibrations $A \rightarrow B$ with $B$ having a countable homogeneous basis. We have to prove that $q$ is a fibration, that is the lifting problem in 
$$
\xymatrix{
C \ar[r]^f  \ar[d]_i & X \ar[d]^q \\
D \ar[r] \ar@{.>}[ru]& Y
}
$$

where the morphism $i$ is an arbitrary acyclic cofibration, has a solution.
This problem is solved by using Zorn's lemma.
Let $\Omega$ be the set of pairs $(\bar{D}, g)$ where $\bar{D}$ fits into a sequence of acyclic cofibrations $C \stackrel{\subseteq}{\longrightarrow} \bar{D} \stackrel{\subseteq}{\longrightarrow} D$ in \textbf{DGcoAlg} and $g \colon \bar{D} \rightarrow X$ is a solution to the restricted lifting problem.
We order the set $\Omega$ by setting $(\bar D_1, g_1) \preceq (\bar D_2, g_2)$ if $\bar D_1 \subseteq \bar D_2$ and ${g_2}_{\mid_{\bar D_1}} = g_1$.
The poset $(\Omega, \preceq)$ is non empty since it contains $(C, f)$. Moreover, any chain in $(\Omega, \preceq)$ has an upper bound given by the union. Hence, by Zorn's lemma $(\Omega, \preceq)$ posses a maximal element $(E, g)$.
We now prove that $E=D$. Let $x$ be a homogeneous element of $D$. Since $E \stackrel{\subseteq}{\longrightarrow} D$ is an acyclic cofibration, Lemma \ref{lem:homogeneous} guarantees the existence of a subcoalgebra $B$ such that $x \in B$, $B$ has a countable homogeneous basis and $E \cap B \rightarrow B$ is an acyclic cofibration. Thus, the induced lifting problem in 
$$
\xymatrix{
E \cap B \ar[r]^{\subseteq} \ar[d] & E  \ar[r]^g & X \ar[d]^q \\
B  \ar@{.>}[rru]_l \ar[r]& D \ar[r] & Y
}
$$
has a solution by hypothesis. 
Note that the object $E+B$ is a pushout of the diagram $E \leftarrow E \cap B \rightarrow B$ in \textbf{DGcoAlg} because the forgetful functor $U_d$ creates colimits by Proposition \ref{prop:cocomplete}. Hence, the map $g$ can be extended as $\bar g \colon E + B \rightarrow X$ with $\bar g(e+b) = g(e)+l(b)$.
Moreover, the resulting Mayer-Vietoris sequence
$$\cdots \rightarrow H_n(E \cap B) \rightarrow H_n(E) \oplus H_n(B) \rightarrow H_n(E+B) \rightarrow H_{n-1}(E \cap B) \rightarrow \cdots $$
shows that $E \rightarrow E+B$ is an acyclic cofibration. Both facts imply that $(E+B, \bar g) \in (\Omega, \preceq)$. Since $(E, g)$ is a maximal element, we deduce that $E = E+B$ and therefore $x \in E$ as required.
\end{proof}

\begin{prop} \cite[Lemma 2.7]{GG99}. \label{prop:mc5ii} \\
A morphism $C \rightarrow D$ of differential graded coalgebras can be factored as $C \stackrel{i}{\longrightarrow} X \stackrel{p}{\longrightarrow} D$ where $i$ is an acyclic cofibration and $p$ a fibration. Furthermore $i$ is in the class of morphisms generated by the acyclic cofibrations $A \rightarrow B$ such that $B$ has a countable homogeneous basis.
\end{prop}

\begin{proof}
By Lemma \ref{lem:small}, any object in \textbf{DGcoAlg} is small relative to the whole category. Moreover the category \textbf{DGcoAlg} is cocomplete. Hence, in light of Lemma \ref{gen-triv-cof} the small object argument applies by using Bousfield-Smith cardinality argument. We refer to \cite[Section 11]{Bou75} where this latter argument first appeared and to the proof of \cite[Theorem 2.3.13]{Hov99} where this technique is applied to an example similar to ours.
\end{proof}

We can now sum up the various ingredients for providing \textbf{DGcoAlg} with a model category structure.

\begin{proof}[\textbf{Proof of Theorem \ref{defn:model}}]
The axiom \textbf{MC1} is given by Proposition \ref{prop:cocomplete}. The axioms \textbf{MC2} and \textbf{MC3} follow by inspection. Parts (i) and (ii) of the axiom \textbf{MC5} are proven respectively in Proposition \ref{prop:mc5i} and in Proposition \ref{prop:mc5ii}. Part (ii) of the axiom \textbf{MC4} is the definition of fibrations as stated in Theorem \ref{defn:model}. It remains to prove the part (i) of the axiom \textbf{MC4}. Let $p \colon C \rightarrow D$ be an acyclic fibration. By Proposition \ref{prop:mc5i}, factor $p$ as $C \stackrel{j}{\longrightarrow} X \stackrel{q}{\longrightarrow} D$ where $j$ is a cofibration and $q$ is an acyclic fibration with the right lifting property with respect to all cofibrations. Note that $j$ is also a weak equivalence by the axiom \textbf{MC2}. Thus the lifting problem
$$
\xymatrix{
C \ar@{=}[r] \ar[d]_j & C \ar[d]^p \\
X \ar@{.>}[ur]_l \ar[r]_q & D
}
$$
has a solution since $p$ has the right lifting property with respect to all acyclic cofibrations. Hence, the map $p$ is a retract of $q$ and any lifting problem 
$$
\xymatrix{
A \ar[r] \ar[dd]^f & C \ar[d]_j \\
                & X \ar[d]^q \ar@/_/[u]_l \\
B \ar[r] \ar@{.>}[ruu]  \ar@/_/[ru]    & D								
}
$$
with a cofibration $f$ has a solution given by the composite of the curved arrows. 
\end{proof}

We end this section with the following useful result that is a consequence of Lemma \ref{gen-triv-cof} and Bousfield-Smith cardinality argument. 

\begin{prop} \cite[Lemma 2.9]{GG99}. \label{cof-gen} \\
The category of differential graded coalgebras \textbf{DGcoAlg} is cofibrantly generated: 
\begin{enumerate}
\item The set of generating cofibrations is given by cofibrations $A\rightarrow B$ in \textbf{DGcoAlg} with $B$ finite dimensional.
\item The set of generating acyclic cofibrations is given by acyclic cofibrations $A \rightarrow B$ in \textbf{DGcoAlg} with $B$ having a countable homogeneous basis.
\end{enumerate}
\end{prop}

\subsection{A model category structure on \textbf{ScoAlg}}
The category of cocommutative simplicial coalgebras has a model structure due to Goerss in \cite[Section 3]{Goe95}. But one can adapt the arguments therein to the non-cocommutative case as well.
A map $f$ in \textbf{ScoAlg} is a
weak equivalence if $\pi_{\ast}f \cong H_\ast Nf$ is an isomorphism,
a cofibration if it is a levelwise inclusion,
and a fibration if it has the right lifting property with respect to trivial cofibrations.
Goerss' main line of argumentation remains unchanged for the case of 
non-cocommutative simplicial coalgebras.
The only slight difference is concerned with Lemma \ref{lemGoe95} recalled below.

\begin{lem} \label{lemsimplicialcoal-vect}
The forgetful functor $U_s$ from the category of simplicial coalgebras to the category of simplicial vector spaces has a right adjoint $S_s$.
\end{lem}

\begin{proof}
The functor $S_s$ is obtained by extending degreewise the cofree coalgebra functor $S$ from the category of vector spaces to the category of non-cocommutative coalgebras as constructed in \cite[Theorem 6.4.1]{Swe69}.
\end{proof}

\begin{lem}\cite[Lemma 3.5]{Goe95}. \label{lemGoe95} \\
Let $f \colon C \rightarrow D$ be a morphism of coalgebras. Then, $f$ can be factored as $f= p \circ i$
$$ C \stackrel{i}{\longrightarrow} X \stackrel{p}{\longrightarrow} D$$
where $i$ is a cofibration and $p$ is an acyclic fibration.
\end{lem}

In the proof of this lemma we may replace the cocommutative cofree functor by its non-cocommutative version
 $S_s \colon \textbf{SVct} \longrightarrow \textbf{ScoAlg}$. 
In this way, Goerss' arguments transfer to the non-cocommutative setting since only the cofreeness property is required.

\section{A comparison of coalgebra categories}

\subsection{Dold-Kan functors for coalgebras}
The Dold-Kan correspondence asserts that \textbf{SVct} and \textbf{DGVct} are equivalent. 
This equivalence of categories is achieved by the normalization functor $N$ and its inverse $\Gamma$. For a fuller description of these functors, we refer to \cite[8.8.4]{Wei94}.
Moreover, it is well-known that the normalized version of the shuffle map $\nabla \colon NA \otimes NB \rightarrow N(A \widehat{\otimes} B)$ makes the normalization functor $N$ lax monoidal while the Alexander-Whitney map $AW \colon N(A \widehat{\otimes} B) \rightarrow NA \otimes NB$ makes it oplax monoidal. We refer to 
\cite[Chapter VIII, Section 8, Corollaries 8.6, 8.9]{Mac63} for a detailed description of the Alexander-Whitney and shuffle maps and their normalized versions.
From these observations we derive that the normalization functor $N$ and its inverse $\Gamma$ pass to the level of comonoids as stated in the propositions below.

\begin{prop}
If $(A,\Delta _A,\varepsilon _A)$ is a simplicial coalgebra,
then $(NA,\Delta _{NA},\varepsilon _{NA})$ is a differential graded coalgebra with a comultiplication  given by the composition
$$
\xymatrix{
NA  \ar [rr]^{N(\Delta _A)}       & &  N(A \widehat \otimes A) \ar[rr]^{AW_{A,A}}      & &  NA \otimes NA
}
$$
and counit  given by $N (\varepsilon_A)$.
\end{prop}

\begin{prop}
If $(B,\Delta _B ,\varepsilon _B)$ is a differential graded coalgebra, \\
then $(\Gamma B,\Delta _{\Gamma B},\varepsilon _{\Gamma B})$ is a simplicial coalgebra
with a  comultiplication $\Delta _{\Gamma B}$ given  by the following composition

$$
\xymatrix{
\Gamma B \ar[rr]^{\Gamma (\Delta_B)} \ar[ddrrrrr]_{\Delta _{\Gamma B}} &   & \Gamma (B \otimes B) \ar[rrr]^{\Gamma (\varepsilon _{B}^{-1} \otimes \varepsilon _{B}^{-1} )} & &  & \Gamma (N \Gamma B \otimes  N \Gamma B) \ar[d]^{\Gamma ( \nabla _{\Gamma B, \Gamma B})}                \\
                                                                         &   &                                                                                                 & &   & \Gamma N(\Gamma B \widehat \otimes \Gamma B) \ar[d]^{\eta _{ \Gamma B \widehat \otimes \Gamma B}^{-1}} \\
                                                                         &   &                                                                                                 & &   & \Gamma B \widehat \otimes \Gamma B
}
$$
and  counit  given by $\Gamma (\varepsilon_B).$
\end{prop}
Notice that both propositions and their proofs are dual to results in \cite[Section 2.3]{SS03}. For detailed proofs, we refer to \cite[Section 5.1]{Sor10}.\\\\
Dually to \cite[Section 2.4]{SS03}, we point out that the coalgebra-valued functors $N$ and $\Gamma$ are not adjoint. This failure is made more precise in Remark \ref{rem:failure} below. 

\begin{lem}
The adjunction counit $\varepsilon \colon N \Gamma \longrightarrow \mathrm{Id}$ is a comonoidal transformation.
Let $\psi_{X,Y}$ be the composition of natural maps 
$$
\psi _{X,Y} = \eta ^{-1}_{\Gamma X \widehat \otimes  \Gamma Y} \circ \Gamma (\nabla _{\Gamma X,\Gamma Y}) \circ \Gamma (\varepsilon ^{-1}_{X} \otimes \varepsilon^{-1}_{Y}).
$$
Then the diagram
$$
\xymatrix{
N\Gamma (X \otimes Y) \ar[rr]^{N(\psi _{X,Y})} \ar[d]_{\varepsilon _{X \otimes Y}} &   &   N (\Gamma X \widehat \otimes  \Gamma Y) \ar[rr]^{AW_{\Gamma X, \Gamma Y}}  &  & N \Gamma X \otimes N \Gamma Y \ar[d]^{\varepsilon _X \otimes  \varepsilon _Y } \\
X \otimes Y \ar@{=}[rrrr]                    &   &                                                    &  & X \otimes Y
}
$$
commutes for every $X$,$Y$ \textit{in} \textbf{DGcoAlg}.
\end{lem}

\begin{prop}
The functor $\Gamma \colon \mbox{\textbf{DGcoAlg}} \longrightarrow \mbox{\textbf{ScoAlg}}$
is full and faithful and respects coalgebra structures.
Moreover, the composite endofunctor $N\Gamma$ is naturally isomorphic to the identity functor on the level of comonoids categories.
\end{prop}

\begin{rem} \label{rem:failure}
The unit $\eta \colon \mathrm{Id} \longrightarrow \Gamma N$ does not have good comonoidal properties.
More precisely, there are objects $X$ and $Y$ in the category \textbf{ScoAlg} so that the diagram
$$
\xymatrix{
X \widehat \otimes Y  \ar@{=}[rrrr]    \ar[d]_{\eta _{X \widehat \otimes Y}}   &    &     &   &   X \widehat \otimes Y  \ar[d]^{\eta _{X} \widehat \otimes \eta _{Y} }  \\
\Gamma N (X \widehat \otimes Y) \ar[rr]_{\Gamma (AW_{X,Y})}     &    &     \Gamma (NX \otimes NY) \ar[rr]_{\psi _{NX,NY}} &     & \Gamma NX \widehat \otimes \Gamma NY
}
$$
does not commute.
Indeed, consider for example $X = Y = \Gamma (\textbf{Z}[1])$ as in \cite[Remark 2.14]{SS03}.	
Since $N$ is left inverse to $\Gamma $ by the previous proposition, one has
$$NX = NY = N \Gamma (\textbf{Z}[1]) \cong \textbf{Z}[1] \hspace{.2cm} \mathrm{and} \hspace{.2cm} NX \otimes NY \cong  \textbf{Z}[1] \otimes \textbf{Z}[1] = \textbf{Z}[2].$$
Therefore the lower composite map in the previous diagram vanishes in degree 1 since
$$[\Gamma (NX \otimes NY)]_1 = [\Gamma (\textbf{Z}[2])]_1 = 0.$$
But in degree 1,  the right map $\eta _X \widehat \otimes \eta _Y$ is an isomorphism between free abelian groups of rank one since
$$[\Gamma (NY)]_1 \cong  [Y]_1 \cong \textbf{Z} \cong  [X]_1 \cong  [\Gamma (NX)]_1.  $$
\end{rem}

\subsection{Quillen adjunctions for coalgebras}
In this section we consider Quillen's setting of model categories and we investigate whether the coalgebra-valued functors $N$ and $\Gamma$ fit into this framework. In addition to the model structures on the categories of coalgebras proven in Section \ref{modcatdgs} we recall the model category structures of the categories of vector spaces involved in this work.
\begin{enumerate}

\item The category \textbf{DGVct} has a model category structure (see \cite[Chapter I, Example B]{Qui67}). A map $f$ in the category \textbf{DGVct} is a weak equivalence if $H_{\ast }f$ is an isomorphism, a cofibration if for each $n \geq 0$, $f_n$ is injective,
and a fibration if for each $n \geq 1$, $f_n$ is surjective. In \cite[Section 7.18]{DS95} it is proven that \textbf{DGVct} is cofibrantly generated. The generating acyclic cofibrations are given by $\big\{0 \rightarrow \mathbb{D}^n \mid n \geq 1 \big\}$ and the generating cofibrations by $\big\{\mathbb{S}^{n-1} \rightarrow \mathbb{D}^n \mid n \geq 0 \big\}$ with $\mathbb S^{-1}$ being the zero chain complex. \\

\item The category \textbf{SVct} has a model category structure (see \cite[II.4, II.6]{Qui67}). A map $f$ in the category \textbf{SVct} is a weak equivalence if $\pi_{\ast}f$ is an isomorphism and a fibration if it is a fibration in the underlying category of simplicial sets.
Since \textbf{DGVct} is cofibrantly generated, one can deduce 
that \textbf{SVct} is cofibrantly generated by applying the transfer result by Crans in \cite[Section 3]{Cra95} to the adjoint pair $(\Gamma, N)$ provided by the Dold-Kan correspondence. Hence in the category \textbf{SVct}, the generating acyclic cofibrations are given by $\big\{0 \rightarrow \Gamma(\mathbb{D}^n) \mid n \geq 1 \big\}$ and the generating cofibrations by
 $\big\{ \Gamma(\mathbb{S}^{n-1}) \rightarrow \Gamma(\mathbb{D}^n) \mid n\geq 0 \big\}$. 
\end{enumerate}

Our aim now is to compare \textbf{ScoAlg} and \textbf{DGcoAlg} in terms of Quillen adjunctions. 
Recall that the forgetful functor $U_d$ from the category of differential graded coalgebras to the category of differential
graded vector spaces has a right adjoint $S_d$ as proven in Proposition \ref{prop:rightadjoint}. The counterpart result for the categories of simplicial coalgebras and vector spaces is given by Lemma \ref{lemsimplicialcoal-vect}.
In this way, the situation to be studied may be illustrated in the diagram
$$
\xymatrix{
 (\mbox{\textbf{SVct}}, \widehat\otimes, I(K)) \ar@<+0.5ex>[d]^{S_s} \ar@<+0.5ex>[rr]^{N }       &       &   (\mbox{\textbf{DGVct}}, \otimes, K[0])  \ar@<+0.5ex>[d]^{S_d} \ar@<+0.5ex>[ll]^{\Gamma}   \\
\mbox{\textbf{ScoAlg}} \ar@<+0.5ex>[u]^{U_s}  \ar@<+0.5ex>[rr]^{\widetilde N}      &       &  \mbox{\textbf{DGcoAlg}.}  \ar@<+0.5ex>[u]^{U_d}
}$$

where $\widetilde N$ stands for the coalgebra-valued normalization functor.

\begin{prop}
In the above situation the functor $\widetilde N$ has a right adjoint $R$.
Moreover the adjoint pair $(\widetilde N, R)$ is a Quillen pair.
\end{prop}

\begin{proof}
First, let $V$ be a differential graded vector space and $S_d(V)$ its differential graded cofree coalgebra. We set
$$R(S_d(V)) = S_s \Gamma (V).$$
Indeed, considering the various adjoint pairs $(U_s,S_s)$, $(N, \Gamma)$, $(U_d, S_d)$ and the \- identity $NU_s = U_d\widetilde N$ successively, yields the following bijection
$$\textbf{ScoAlg} \Big(X, RS_dV \Big) \cong \textbf{DGcoAlg} \Big(\widetilde{N}X, S_dV \Big).$$
This means that the functor $R$ is right adjoint to $\widetilde{N}$ for cofree coalgebras.
We now notice that the adjunction $(U_d,S_d)$ defines a monad $S_dU_d$ on the category \textbf{DGcoAlg}.
Thus, if $C$ is a differential graded coalgebra, it can be writtten as the equalizer of the diagram
$$
\xymatrix{
S_dU_d C \ar@<0.5ex>[r]^<<<<<<{d^0} \ar@<-0.5ex>[r]_<<<<<<{d^1} &  S_dU_d S_dU_d C
}
$$
Since the functor $R$ should be a right adjoint it has to preserve limits. Therefore, defining $R(C)$ as  the equalizer of the maps $R(d^0)$ and $R(d^1)$ yields the desired right adjoint.

Finally we observe that the cofibrations and acyclic cofibrations in \textbf{SVct} and \textbf{DGVct} match with those of their respective categories of comonoids \textbf{ScoAlg} and \textbf{DGcoAlg}. Since the functor $N$ is a left Quillen functor the identity $NU_s = U_d\widetilde N$ ensures that the functor $\widetilde N$ is a left Quillen functor.
\end{proof}

\begin{rem}
We mention that instead of $\widetilde N$, we could consider the coalgebra-valued functor $\widetilde \Gamma$. By similar techniques, it is possible to construct a right adjoint functor $R$ to $\widetilde \Gamma$ and show that $(\widetilde \Gamma, R)$ is a Quillen pair.
\end{rem}

\section{Categories of connected coalgebras}
\subsection{Connected differential graded coalgebras} \label{sect:model}
In this section we restrict to \textit{connected} differential graded objects. An object $V$ in \textbf{DGVct} is connected if $V_0 =0$ and we denote by \textbf{DGVct}$_\mathrm{c}$ the full category of connected differential vector spaces. With mild changes, the category \textbf{DGVct}$_\mathrm{c}$ inherits a model category structure from \textbf{DGVct}. Indeed, \textbf{DGVct}$_\mathrm{c}$ has limits and colimits constructed degreewise as in \textbf{DGVct}. Moreover, the model category factorization axioms may be performed as in \cite[Section 1.3]{GS07}, but discarding this time objects such as the $0$-sphere $\mathbb{S}^0$, $\mathbb D^{0}$ and $\mathbb{D}^1$. In this way, a map $f$ in the category \textbf{DGVct}$_\mathrm{c}$ is a weak equivalence if $H_{\ast }f$ is an isomorphism, a cofibration if for each $n \geq 1$, $f_n$ is injective,
and a fibration if for each $n \geq 2$, $f_n$ is surjective. \\
An object $C$ in \textbf{DGcoAlg} is connected if $C_0 =K$. We denote by \textbf{DGcoAlg}$_\mathrm{c}$ the category of connected differential graded coalgebras and in the rest of this section we discuss its model category structure. 

Let $V$ be an object in the category \textbf{DGVct}$_\mathrm{c}$. The tensor coalgebra on $V$ is defined by $T'_d(V) = \bigoplus_{n\geq 0} V^{\otimes n}$. Since $V_0 =0$, it follows that $T'_d(V) \cong \prod_{n\geq 0} V^{\otimes n}$. We mention that the structure maps on $T'_d(V)$ are given by 
\begin{eqnarray*}
 \Delta_{T'_d(V)}(v_1 \otimes \cdots \otimes v_n) & = &\sum_{r=0}^{n} (v_1 \otimes \cdots \otimes v_r) \otimes (v_{r+1} \otimes \cdots \otimes v_n)\\
 \Delta_{T'_d(V)}(1)                              & = & 1\otimes 1 \\
 \epsilon_{T'_d(V)}(v_1 \otimes \cdots \otimes v_n)&=& 0 \hspace{.15cm} \mbox{for $n \geq 1$} \hspace{.15cm} \mbox{and} \hspace{.15cm} \epsilon_{T'_d(V)}(1)=  1.
 \end{eqnarray*}

\begin{defn} \cite[Section II.2]{HMS74}.
Let $C$ be a connected differential graded coalgebra.
Then, the functor $I'_d \colon \textbf{DGcoAlg}_\mathrm{c} \rightarrow \textbf{DGVct}_\mathrm{c}$ is defined by 
$$I'_d(C) = C / K[0].$$
In other words, the differential graded vector space $I'_d(C)_n$ is $ C_n$ for $n \geq 1$, and $0$ for $n=0$.
\end{defn}

\begin{prop} \cite[Section II.2]{HMS74}.
The tensor coalgebra functor $$T'_d \colon \textbf{DGVct}_\mathrm{c} \rightarrow \textbf{DGcoAlg}_\mathrm{c}$$ is right adjoint to the functor $I'_d$.
\end{prop}

\subsection{A model category structure for $\textbf{DGcoAlg}_\mathrm{c}$} \label{sect:modelconnected}
This section aims at transferring the model category structure of \textbf{DGcoAlg} to \textbf{DGcoalg}$_c$ by means of the transfer principle as stated for instance in \cite[Sections 2.5, 2.6]{BM03}. 
To this end, Proposition \ref{coalg-adjunction} below exhibits an adjunction between the categories \textbf{DGcoAlg} and \textbf{DGcoalg}$_c$. 
Then Proposition \ref{connectedcocomplete} establishes the completeness and cocompleteness of the category \textbf{DGcoalg}$_c$. 

The rest of the section is intended to construct a functorial path-object for fibrant objects in \textbf{DGcoalg}$_c$. But more is done here through a construction due to \cite[Definition 4.17, Lemma 4.19]{Smi}: we explain how to achieve the acyclic cofibration-fibration factorization axiom in a functorial way for any morphism in \textbf{DGcoalg}$_c$. Indeed, given a morphism $f\colon C\rightarrow D$ in \textbf{DGcoAlg}$_c$, the Lemma \ref{lem:acyclic} below provides a functorial way to factorize $f$ as $C\stackrel{i}{\longrightarrow} X \stackrel{p}{\longrightarrow} D$ with $i$ a cofibration and $p$ a fibration. Then Smith's construction in Section \ref{smithconstruction} builds out of $X$ an object $G$ that is weak equivalent to $C$ proving this way the required axiom \textbf{MC5(ii)}. We close this section with an important Lemma \ref{lem:retract}, useful for proving our main result Theorem \ref{mainresult}.

\begin{prop} \label{coalg-adjunction}
There is an adjunction between the categories $\textbf{DGcoAlg}$ and $\textbf{DGcoAlg}_\mathrm{c}$.
\end{prop}

\begin{proof}
Let $C \in$ \textbf{DGcoAlg} with counit $\epsilon_C \colon C \rightarrow K[0]$. Since $(\epsilon_C)_0 \colon C_0 \rightarrow K$ is a non-zero linear form, it follows that $C_0 \cong \ker (\epsilon_C)_0 \oplus K$. Moreover, the vector space $\ker (\epsilon_C)_0$ is a coideal of the coalgebra $C_0$ by \cite[Theorem 1.4.7]{Swe69}, hence
$C_0 / \ker (\epsilon_C)_0$ has a coalgebra structure and we may define a functor $$F \colon \textbf{DGcoAlg} \rightarrow \textbf{DGcoAlg}_\mathrm{c} \hspace{.1cm} \mbox{by} \hspace{.1cm}
F(C) = C/\left(\ker (\epsilon_C)_0\right)[0]$$ where $\left(\ker (\epsilon_C)_0\right)[0]$ denotes the object of \textbf{DGVct} with $\ker (\epsilon_C)_0$ concentrated in degree $0$. Let $(C, D) \in \textbf{DGcoAlg} \times \textbf{DGcoAlg}_\mathrm{c}$ and $f \colon C \rightarrow D$ a morphism in \textbf{DGcoAlg}. Since $(\epsilon_D)_0 \colon K \rightarrow K$ is an isomorphism and $(\epsilon_D)_0 f_0 = (\epsilon_C)_0$, the following universal problem 
$$
\xymatrix{
C\ar[r]^f \ar[d]_\pi & D \\
F(C) = C/\left(\ker (\epsilon_C)_0\right)[0] \ar@{.>}[ur]
} 
$$
 has a solution. In other words the functor $F$ is left adjoint to $U \colon \textbf{DGcoAlg}_\mathrm{c} \rightarrow \textbf{DGcoAlg}$ that forgets connectedness.
\end{proof}

\begin{prop} \label{connectedcocomplete}
The category of connected differential graded coalgebras is complete and cocomplete.
\end{prop}

\begin{proof}
We first start with limits. A terminal object in \textbf{DGcoAlg}$_\mathrm{c}$ is given by $K[0]$. 
For other limits, we recall that there is an anti-equivalence between the category of differential graded coalgebras and the category of profinite differential graded algebras. Since by the appendix in Section \ref{section:con-alg}, the category of connected differential graded algebras is complete and cocomplete, we derive limits for \textbf{DGcoAlg}$_\mathrm{c}$ by applying the steps given in Proposition \ref{prop:dgcocomplete}. We mention that the usual tensor product $\otimes$ of differential graded coalgebras is \emph{not} the categorical product in \textbf{DGcoAlg}$_\mathrm{c}$ since we do not assume cocommutativity.\\
Regarding colimits, we refer to \cite[Section 1]{Nei78} where similar constructions appear for cocommutative coalgebras. Since as a left adjoint, the functor $I'_d$ must preserve initial objects, we deduce that $K[0]$ is initial in \textbf{DGcoAlg}$_\mathrm{c}$.
Let $f,g \colon C \rightarrow D$ be two maps in \textbf{DGcoAlg}$_\mathrm{c}$. Then, their coequalizer is given by $D / \textrm{im}(f-g).$ This quotient is constructed degreewise and each of its homogeneous parts is in fact a coalgebra by \cite[Proposition 1.4.8]{Swe69}. Finally, if $C$ and $D$ are two objects in \textbf{DGcoAlg}$_\mathrm{c}$, we may form the maps 
$$
K[0] \stackrel{\varphi_C}{\longrightarrow} C \stackrel{i_C}{\longrightarrow} C \oplus D \hspace{.2cm} \textrm{and} \hspace{.2cm} K[0] \stackrel{\varphi_D}{\longrightarrow} D \stackrel{i_D}{\longrightarrow} C \oplus D.
$$
The coproduct of $C$ and $D$ in \textbf{DGcoAlg}$_\mathrm{c}$ is thus given by 
$$C \sqcup D = \left(C \oplus D \right)/ \textrm{im}\left(i_C \circ \varphi_C - i_D \circ \varphi_D \right).$$
Notice that the direct sum is the coproduct of the underlying differential graded vector spaces and that the quotient guarantees the required connectedness condition.
\end{proof}

In order to find a functorial path-object for fibrant objects in \textbf{DGcoAlg}$_\mathrm{c}$, we explain in this paragraph how to achieve a functorial acyclic cofibration-fibration for a given morphism $f \colon C \rightarrow D$ in \textbf{DGcoAlg}$_\mathrm{c}$. This is done by using a technique described in \cite[Definition 4.17, Lemma 4.19]{Smi}. Proceeding this way offers a dual advantage, namely Proposition \ref{modelcatconnected} and Lemma \ref{lem:retract}.

\begin{lem} \label{lem:acyclic}
Let $V$ be an acyclic ($H_\ast V \cong 0$) connected differential graded vector space and $C$ a connected differential graded coalgebra. Then the projection $$C \sqcap T'_d(V)\longrightarrow C$$ is an acyclic fibration.
\end{lem}

\begin{proof}
First note that if $C$ is an object of \textbf{DGcoAlg}$_c$ then its linear dual $C^\ast=\left\{\textbf{Vct}(C_n, K)\right\}$ becomes a connected differential non-positively graded algebra. 
We also denote by $T_{C^\ast}$ the tensor $C^\ast$-algebra functor defined by 
$$T_{C^\ast}(M) = C^\ast \oplus \bigoplus_{n\geq 1} \underbrace{M\otimes_{C^\ast} M \otimes_{C^\ast} \cdots \otimes_{C^\ast} M}_{n \; \mathrm{times}}$$ 
for any differential graded $C^\ast$-bimodule $M$. 
With this in hand we obtain
$$\left[C \sqcap T'_d (V)\right]^\ast \cong C^\ast \sqcup \left[T'_d(V)\right]^\ast \cong C^\ast \sqcup T_d (V^\ast) \cong T_{C^\ast}\left(C^\ast \otimes V^\ast \otimes C^\ast \right).$$
The cohomological K\"unneth formula yields
$$H^\ast \left(T_{C^\ast}\left(C^\ast \otimes V^\ast \otimes C^\ast \right)\right) \cong H^\ast \left(C^\ast\right).$$
The required homology isomorphism $H_\ast \left(C \sqcap T'_d(V)\right) \cong H_\ast C$ results from the application of the cohomological universal coefficient theorem.\\
We still have to prove that each projection is a fibration. To do so, we consider the diagram 
$$
\xymatrix{
A \ar[r] \ar[d]_i & C \sqcap T'_d(V) \ar[d] \\
B \ar[r]          & C
}
$$
where $i$ is a cofibration in \textbf{DGcoAlg}$_\mathrm{c}$. By adjointness, this amounts to considering the diagram 
$$
\xymatrix{
I'_d(A) \ar[r] \ar[d]_{I'_d(i)} & V \ar[d] \\
I'_d(B) \ar[r]          & 0
}
$$
in \textbf{DGVct}$_\mathrm{c}$. Since $I'_d(i)$ is a cofibration  and $V \rightarrow 0$ is an acyclic fibration in the model category of \textbf{DGVct}$_\mathrm{c}$, a lift $I'_d(B) \rightarrow V$ exists.
\end{proof}

We recall from \cite[Section 1.5]{Wei94} some useful constructions in \textbf{DGVct}.

\begin{defn} Let $(V,d)$ be an object in the category \textbf{DGVct}.
\begin{itemize}
\item[1.] We denote by $\mathrm{cone}(V)$ the object of \textbf{DGVct} with 
$\left(\mathrm{cone}(V)\right)_n = V_{n-1} \oplus V_n$ and differential given by $d(b, c) = (-d(b), d(c)-b), \hspace{.2cm} b \in V_{n-1}, c \in V_n.$
\item[2.] We denote by $s^{-1}V$ the object of \textbf{DGVct} with
$\left(s^{-1} V\right)_n = V_{n+1} $ and differential given by $-dc, \hspace{.2cm} c \in V_{n+1}.$
\item[3.] We denote by $\tau_{\geq 1} V$ the object of \textbf{DGVct} with
$$
\left(\tau_{\geq 1} V\right)_n =
\left\{\begin{array}{ccl}
0, & \mbox{if $n < 1$,} \\
\ker{(V_1 \rightarrow V_0)}, & \mbox{if $n =1$,}    \\
V_n, & \mbox{if $n > 1$.} 
\end{array}
\right.
$$ 
\end{itemize}
\end{defn}

\begin{rem}
If $V \in$ \textbf{DGVct}$_\mathrm{c}$, then $\mathrm{cone}(V) \in$ \textbf{DGVct}$_\mathrm{c}$. This is not the case for the desuspension $s^{-1}V$ since $\left(s^{-1}V\right)_0 =V_1$. We may avoid this problem by considering the object $\tau_{\geq 1} \left(s^{-1} V\right)$ which lies in \textbf{DGVct}$_\mathrm{c}$.
\end{rem}

\begin{lem} \label{lem:fibration}
Let $V$ be an object of the category \textbf{DGVct}$_\mathrm{c}$. Then, the canonical map $\tau_{\geq 1} \left(s^{-1} \mathrm{cone}(V)\right) \longrightarrow \tau_{\geq 1} \left(V\right) =V$ is a fibration in the model structure of \textbf{DGVct}$_\mathrm{c}$.
\end{lem}

\begin{proof}
Recall from Section \ref{sect:model} that fibrations in \textbf{DGVct}$_\mathrm{c}$ are maps $f$ with $f_n$ surjective for $n \geq 2.$  We may picture our case as follows
$$
\xymatrix{
\cdots \ar[r] & V_2 \oplus V_3 \ar[r] \ar[d] & \ker{\left(V_1 \oplus V_2 \rightarrow V_1\right)} \ar[r] \ar[d] & 0 \ar[d] \\
\cdots \ar[r] & V_2            \ar[r]        & V_1 = \ker{\left(V_1 \rightarrow 0\right)}                                  \ar[r]        & 0.
}
$$
Our given map is a fibration since, for $n\geq 2$, we have canonical projections.
\end{proof}

\begin{lem} \label{lem:isom}
If $C$ is an object of the category \textbf{DGcoAlg}$_\mathrm{c}$, then, $$H_\ast (C) \cong H_\ast (I'_d C) \oplus H_\ast (K[0]).$$
\end{lem}

\begin{proof}
Since the object $K[0]$ is both initial and terminal in the category \textbf{DGcoAlg}$_\mathrm{c}$, we deduce that $K[0]$ splits off the object $C$. Hence $C \cong I'_d(C) \oplus K[0]$ and the result follows. 
\end{proof} 
Furthermore, if $C$ and $D$ are object in \textbf{DGcoAlg}$_\mathrm{c}$ such that $H_\ast (I'_d C) \cong H_\ast (I'_d D)$, then, $H_\ast (C) \cong H_\ast (D)$.

\vspace{.5cm}

\subsection*{Smith's construction} \label{smithconstruction} Let us factor $f$ as $C \stackrel{i_0}{\longrightarrow} D \sqcap T'_d(I'_d C) \stackrel{p_0}{\longrightarrow} D$ where the map $i_0$ is a canonical injection hence a cofibration in \textbf{DGcoAlg}$_\mathrm{c}$ and $p_0$ is a fibration. 
Then, construct the maps $\left(i_n\right)_{n\geq 1}$ and $\left(p_n\right)_{n\geq 1}$ as displayed in the diagram
$$
\xymatrix{
C \ar[rr]|-{i_0} \ar@/_/[drr]|-{i_1} \ar@/_/[ddrr]|-{i_n} \ar@/_/[dddrr]|-{i_{n+1}} \ar@/_/[ddddrr]|-{i_\infty}&    &  G(0) = D \sqcap T'_d(I'_d C) \ar[r]|-{p_0} &   D \\
          &    &  G(1) \ar[u]_{p_1}  &     \\
          &    &  G(n) \ar@{--}[u]  &     \\
          &    &  G(n+1) \ar[u]_{p_{n+1}}  &     \\
          &    &  \lim_n G(n) \ar@{--}[u]  &     \\
}
$$
in which
\begin{enumerate}
\item for $n\geq 0$, the object $G(n+1)$ is a pullback of $$G(n) \longrightarrow T'_d \Big[\tau_{\geq 1} \left(I'_d H(n)\right)\Big] \longleftarrow T'_d \Big[\tau_{\geq 1}\left(s^{-1}\mathrm{cone}\left(I'_d H(n)\right)\right)\Big]$$
where $H(n)$ is a pushout of $G(n) \longleftarrow C \stackrel{\epsilon_C}{\longrightarrow} K[0]$ 
\item the maps $\left(p_n\right)_{n\geq 1}$ form a tower of fibrations. In fact, as pullbacks, they are built out of \textbf{DGVct}$_\mathrm{c}$ fibrations $\tau_{\geq 1}\left(s^{-1}\mathrm{cone}\left(I'_d H(n)\right)\right) \longrightarrow \tau_{\geq 1}\left(I'_d H(n)\right)$ as shown in Lemma \ref{lem:fibration}.
\item the maps $\left(i_n\right)_{n\geq 1}$ are cofibrations since by induction the compositions $p_n \circ i_n$ are injections.
\end{enumerate}

The following result is analogue to \cite[Lemma 4.19]{Smi}.

\begin{lem}
The resulting object $\lim_n G(n) \in $ \textbf{DGcoAlg}$_\mathrm{c}$ described above is weakly equivalent to $C$.
\end{lem}

\begin{proof}
Following \cite[Proof of Lemma 4.19]{Smi} the idea is to factorize the identity map on $\left(\lim_n G(n) \right) / i_\infty \left(C\right)$ through the acyclic object  $T'_d \left[ \tau_{\geq 1} s^{-1} \mathrm{cone} \lim_n I'_d H(n)\right]$. To this end, we observe the following category-theoretic facts. We will denote by $A\times_C B$ a pullback of the diagram $A\rightarrow C \leftarrow B$.

\begin{enumerate} 
\item As a pushout of the diagram $\lim_n G(n) \stackrel{i_\infty}{\longleftarrow} C \longrightarrow K[0]$, we have that $$\lim_n H(n) \cong \left(\lim_n G(n) \right) / i_\infty \left(C\right) = \lim_n \left( G(n)  / i_n \left(C\right) \right).$$
\item By construction of $G(n)$ in the preceding paragraph, we have
$$
\lim_n G(n) = \lim_n G(n) \times_{T'_d \Big[\tau_{\geq 1} \left(I'_d H(n)\right)\Big]} T'_d \Big[\tau_{\geq 1}\left(s^{-1}\mathrm{cone}\left(I'_d H(n)\right)\right)\Big].
$$
\item Let us consider the canonical projection 
$$ \lim_n G(n) \times_{T'_d \Big[\tau_{\geq 1} \left(I'_d H(n)\right)\Big]} T'_d \Big[\tau_{\geq 1}\left(s^{-1}\mathrm{cone}\left(I'_d H(n)\right)\right)\Big] \longrightarrow \lim_n G(n).$$ 
Since $\tau_{\geq 1}\left(s^{-1}\mathrm{cone}\left(I'_d H(n)\right)\right) \longrightarrow \tau_{\geq 1}\left(I'_d H(n)\right)$ is a fibration in 
\textbf{DGVct}$_\mathrm{c}$, there is a map such that the composite 
$$\tau_{\geq 1}\left(I'_d H(n)\right) \longrightarrow \tau_{\geq 1}\left(s^{-1}\mathrm{cone}\left(I'_d H(n)\right)\right) \longrightarrow \tau_{\geq 1}\left(I'_d H(n)\right)$$ is the identity. Applying the cofree functor $T'_d$ to this map yields a map that splits the canonical projection given above.
\end{enumerate}

Consequently, we obtain a map $$\lim_n G(n) \longrightarrow \lim_n G(n) \times_{T'_d \Big[\tau_{\geq 1} \left(I'_d H(n)\right)\Big]} T'_d \Big[\tau_{\geq 1}\left(s^{-1}\mathrm{cone}\left(I'_d H(n)\right)\right)\Big]$$ that splits the projection map in (3).
Furthermore, we may produce the diagram
$$
\xymatrix{
\left(\lim_n G(n)\right) / i_\infty (C) \ar[d]   \\
\left( \lim_n G(n) \times_{T'_d \Big[\tau_{\geq 1} \left(I'_d H(n)\right)\Big]} T'_d \Big[\tau_{\geq 1}\left(s^{-1}\mathrm{cone}\left(I'_d H(n)\right)\right)\Big]\right) / i_\infty (C)\ar[d] \\
\left(\lim_n G(n)\right) / i_\infty (C) \times_{ T'_d \Big[\tau_{\geq 1} \left(I'_d H(n)\right)\Big]} T'_d \Big[\tau_{\geq 1}\left(s^{-1}\mathrm{cone}\left(I'_d H(n)\right)\right)\Big]  \ar@{=}[d] \\
\lim_n \left(G(n) / i_n (C)\right) \times_{ T'_d \Big[\tau_{\geq 1} \left(I'_d H(n)\right)\Big]} T'_d \Big[\tau_{\geq 1}\left(s^{-1}\mathrm{cone}\left(I'_d H(n)\right)\right)\Big]   \ar@{=}[d] \\
\lim_n H(n) \times_{ T'_d \Big[\tau_{\geq 1} \left(I'_d H(n)\right)\Big]} T'_d \Big[\tau_{\geq 1}\left(s^{-1}\mathrm{cone}\left(I'_d H(n)\right)\right)\Big]  \ar[d] \\
\lim_n H(n) = \left(\lim_n G(n)\right) / i_\infty (C)
}
$$
that factorizes $\left(\lim_n G(n) \right) / i_\infty \left(C\right) \stackrel{\mathrm{Id}}{\longrightarrow} \left(\lim_n G(n) \right) / i_\infty \left(C\right)$. 
Since 
$$
\begin{array}{l}
\lim_n H(n) \times_{ T'_d \Big[\tau_{\geq 1} \left(I'_d H(n)\right)\Big]} T'_d \Big[\tau_{\geq 1}\left(s^{-1}\mathrm{cone}\left(I'_d H(n)\right)\right)\Big]   \\
\subseteq \lim_n T'_d\left(\lim_n \Big[\tau_{\geq 1} \left(I'_d H(n)\right)\Big]\right) \times_{ T'_d \Big[\tau_{\geq 1} \left(I'_d H(n)\right)\Big]} T'_d \Big[\tau_{\geq 1}\left(s^{-1}\mathrm{cone}\left(I'_d H(n)\right)\right)\Big]  \\
= T'_d \Big[\tau_{\geq 1}\left(s^{-1}\mathrm{cone}\left(I'_d H(n)\right)\right)\Big]
\end{array}
$$
we deduce that the identity map $\left(\lim_n G(n) \right) / i_\infty \left(C\right) \longrightarrow \left(\lim_n G(n) \right) / i_\infty \left(C\right)$ may be factorised through $T'_d \left[ \tau_{\geq 1} s^{-1} \mathrm{cone} \lim_n I'_d H(n)\right]$. This latter object is acyclic because $\tau_{\geq 1} s^{-1} \mathrm{cone} \lim_n I'_d H(n)$ is canonically acyclic and homology commutes with the functor $T'_d$ by K\"unneth theorem. It follows that $H_\ast \Big[I'_d \left(\left(\lim_n G(n) \right) / i_\infty \left(C\right)\right) \Big]=0.$ Finally, the long exact sequence of homology resulting from the short exact sequence 
$$0 \longrightarrow I'_d C \longrightarrow I'_d\left(\lim_n G(n)\right) \longrightarrow I'_d\left(\left(\lim_n G(n) \right) / i_\infty \left(C\right)\right)\longrightarrow 0$$ gives the required weak equivalence by using Lemma \ref{lem:isom}.
\end{proof}

\begin{prop} \label{prop:simth}
Any morphism $f \colon C \rightarrow D$ in \textbf{DGcoAlg}$_\mathrm{c}$ can be factored as $$C \stackrel{i}{\longrightarrow} G \stackrel{p}{\longrightarrow} D$$ with $i$ an acyclic cofibration and $p$ a fibration.
\end{prop}

\begin{proof}
Setting $i= i_\infty$ and taking $p$ to be induced by the $p_i$'s give the required factorization.
\end{proof}

\begin{prop} \label{modelcatconnected}
The category of connected differential graded coalgebras has a model category structure. A morphism $f \in \textbf{DGcoAlg}_\mathrm{c}$ is a weak equivalence (resp. a fibration) if $U(f) =f$ is a weak equivalence (resp. a fibration) in \textbf{DGcoAlg}.
\end{prop}

\begin{proof}
By Proposition \ref{connectedcocomplete}, the category \textbf{DGcoAlg}$_\mathrm{c}$ is (co)complete. Moreover, Proposition \ref{prop:simth} 
endows \textbf{DGcoAlg}$_\mathrm{c}$ with a fibrant replacement functor and a functorial path-object $X \stackrel{i}{\longrightarrow} \mathrm{Path}(X) \stackrel{p}{\longrightarrow} X \sqcap X$ for any fibrant object $X \in$ \textbf{DGcoAlg}$_\mathrm{c}$. Finally, since the category \textbf{DGcoAlg} is cofibrantly generated by 
Proposition \ref{cof-gen}, the transfer principle applies and \textbf{DGcoAlg}$_\mathrm{c}$ inherits a model category structure.
\end{proof}

\begin{lem} \label{lem:retract}
Let $C$ be a fibrant connected differential graded coalgebra and  $$ C \stackrel{i}{\longrightarrow} G \stackrel{p}{\longrightarrow} K[0]$$ be the factorization of $\epsilon_C \colon C \rightarrow K[0]$ as in Proposition \ref{prop:simth}.
Then, $C$ is a retract of $G$. Moreover, $G$ is a cofree connected differential graded coalgebra.
\end{lem}

\begin{proof}
Since $C$ is fibrant, the counit $\epsilon_C$ is a fibration. It follows that a lift exists in the diagram
$$
\xymatrix{
C \ar@{=}[r] \ar[d]_i & C \ar[d]^{\epsilon_C}\\
G \ar[r]_p \ar@{.>}[ru]& K[0] 
}
$$
and therefore that $C$ is a retract of $G$.\\
For the second statement, note first that $G(0)$ is cofree in \textbf{DGcoAlg}$_\mathrm{c}$ since $$G(0) = T'_d\left(I'_d C\right) \sqcap K[0] = T'_d\left(I'_d C \right) \sqcap T'_d\left(0\right) \cong T'_d\left(I'_d C \times 0\right) \cong T'_d\left(I'_d C\right).$$
Then, the first step of Proposition \ref{prop:simth} computes the object $H(0)$ as a pushout
$$
\xymatrix{
C \ar[r] \ar[d] &  K[0] \ar[d] \\
G(0)\cong T'_d\left(I'_d C\right) \ar[r] & H(0)
}
$$
The second step gives $G(1)$ as a pullback in the diagram 
$$
\xymatrix{
G(1) \ar[rr] \ar[d] & &  \ar[d] T'_d \Big[\tau_{\geq 1}\left(s^{-1}\mathrm{cone}\left(I'_d H(0)\right)\right)\Big]\\
G(0)\cong T'_d\Big[I'_d C\Big] \ar[rr] && T'_d \Big[\tau_{\geq 1} \left(I'_d H(0)\right)\Big]
}
$$
Then, the map $C \rightarrow H(0)$ yields the map $I'_d C \rightarrow  I'_d H(0) = \tau_{\geq 1} (I'_d H(0))$. By the universal property of the functor $T'_d$, we deduce that the map
$$G(0)\cong T'_d\Big[I'_d C\Big] \longrightarrow T'_d \Big[\tau_{\geq 1} \left(I'_d H(0)\right)\Big]$$ 
is of the form
$T'_d \Big[I'_d C \longrightarrow \tau_{\geq 1} \left(I'_d H(0)\right)\Big].$
Hence, the object $G(1)$ is cofree as a pullback of cofree objects and maps induced by maps in the category \textbf{DGcoAlg}$_\mathrm{c}$. By the same arguments, the objects $G(n)$ are cofree and
$G = \lim_n G(n)$ is cofree as required.
\end{proof}

\subsection{Connected simplicial coalgebras}

This section is the simplicial counterpart of the previous one.
An object $V$ in \textbf{SVct} is connected if $V_0 = 0$. 
A connected simplicial coalgebra $C$ is an object in \textbf{ScoAlg} with $C_0=K$. We denote by \textbf{SVct}$_\mathrm{c}$ and \textbf{ScoAlg}$_\mathrm{c}$ the categories of connected simplicial vector spaces and coalgebras.

\begin{lem} \label{lem:split}
Let $C$ be a connected simplicial coalgebra. Then, the constant simplicial object $I(K)$ splits off the object $C$.
\end{lem}

\begin{proof}
In \textbf{ScoAlg}$_\mathrm{c}$, $I(K)$ is both initial and terminal. Consequently, the canonical map
$i \colon I(K) \rightarrow C$ is an injection and $I(K)$ splits off $C$.
\end{proof}

The following definition is a consequence of the previous Lemma \ref{lem:split}.

\begin{defn} 
Let $C$ be a connected simplicial coalgebra. Then, there is a functor $I'_s \colon \textbf{ScoAlg}_\mathrm{c} \rightarrow \textbf{SVct}_\mathrm{c}$ defined by $I'_s(C)= C / I(K).$
\end{defn}

\begin{prop}
The functor $I'_s \colon \textbf{ScoAlg}_\mathrm{c} \rightarrow \textbf{SVct}_\mathrm{c}$ has a right adjoint defined by 
$$T'_s(W) = \bigoplus_{n \geq 0} W^{\widehat{\otimes}n} = I(K) \oplus W \oplus \cdots \oplus W^{\widehat{\otimes}n} \oplus \cdots $$ for any object $W$ in the category  \textbf{SVct}$_{\mathrm c}$.
\end{prop}

\begin{proof}
The required adjunction is obtained by considering the following bijections with $C \in$ \textbf{ScoAlg}$_\mathrm{c}$ and $W \in$ \textbf{SVct}$_\mathrm{c}$
\begin{eqnarray*}
\textbf{SVct}_\mathrm{c}(I'_s C, V) & \cong & \textbf{DGVct}_\mathrm{c}(NI'_sC, NW) \\
                                    & \cong &  \textbf{DGVct}_\mathrm{c}(I'_dNC, NW)\\
																		& \cong & \textbf{DGcoAlg}_\mathrm{c}(NC, T'_d(NW))\\
																		& \cong & \textbf{ScoAlg}_\mathrm{c}(C, RT'_d(NW)) \\
																		& \cong &  \textbf{ScoAlg}_\mathrm{c}(C, T'_s\Gamma (NW)) \\
																		& \cong & \textbf{ScoAlg}_\mathrm{c}(C, T'_s(W)).
\end{eqnarray*}
Note that the third bijection comes from considering the diagram
$$
\xymatrix{
NW & I'_dNC \ar[l] \ar@{.>}[ldd] & NC \ar[l] \\
NT'_s(W) \ar[u]^{N\pi} &  &  \\
T'_d(NW) \ar[u]^{\nabla} \ar[u]
}
$$
where $\nabla \colon T'_d(NW) \rightarrow NT'_s(W)$ is the normalized shuffle map and $\pi \colon T'_s(W) \rightarrow W$ the canonical projection.
\end{proof}

\begin{prop} \label{prop:cocomplete}
The category of connected simplicial coalgebras is complete and cocomplete.
\end{prop}

\begin{proof}
 For limits in \textbf{ScoAlg}$_\mathrm{c}$ it suffices to extend degreewise the construction in \cite[Theorem 1.1]{Ago11} for the category of coalgebras over fields. Hence, a terminal object in \textbf{ScoAlg}$_\mathrm{c}$ is given by $I(K)$, the constant simplicial coalgebra.
 Notice that, since the field $K$ is a terminal object in \textbf{coAlg}, the product $K \sqcap K$ is isomorphic to $K$. This ensures that the products $C \sqcap D$ of objects in \textbf{ScoAlg}$_\mathrm{c}$ is again connected since $\left(C \sqcap D\right)_0 = C_0 \sqcap D_0 = K \sqcap K = K$.
 
Colimits in \textbf{ScoAlg}$_\mathrm{c}$ are formed in the same way as for \textbf{DGcoAlg}$_\mathrm{c}$. In this way, an initial object is given by $I(K)$. If $f,g \colon C \rightarrow D$ are two maps in \textbf{ScoAlg}$_\mathrm{c}$, their coequalizer is given by $D / \textrm{im}(f-g)$. Finally, if $C$ and $D$ are two objects in \textbf{ScoAlg}$_\mathrm{c}$, we may form the maps 
$$
I(K) \stackrel{\varphi_C}{\longrightarrow} C \stackrel{i_C}{\longrightarrow} C \oplus D \hspace{.2cm} \textrm{and} \hspace{.2cm} I(K) \stackrel{\varphi_D}{\longrightarrow} D \stackrel{i_D}{\longrightarrow} C \oplus D.
$$
The coproduct of $C$ and $D$ in \textbf{ScoAlg}$_\mathrm{c}$ is then given by 
$$C \sqcup D = C \oplus D / \textrm{im}\left(i_C \circ \varphi_C - i_D \circ \varphi_D \right).$$
Notice that the direct sum is taken degreewise and that the quotient guarantees the connectedness condition.
\end{proof}

With the above-mentioned facts, the category \textbf{ScoAlg}$_\mathrm{c}$ is endowed with a model category structure exactly as in \cite[Section 3]{Goe95}.

\subsection{A Quillen equivalence for connected coalgebras}
In this section, we improve the Quillen adjunction  $(\widetilde{N}, R)$ to a Quillen equivalence.
We were not able to check Hovey's criterion (see \cite[Corollary 1.3.16]{Hov99}) for arbitrary fibrant differential graded coalgebras. However, connectedness is a condition that guarantees such a criterion which yields a Quillen equivalence. 

\begin{lem}
There is an equivalence between the category of connected simplicial vector spaces and the category of connected differential graded vector spaces.
\end{lem}

\begin{proof}
We notice that the restriction of the normalization functor $N \colon \textbf{SVct$_{\mathrm{c}}$} \rightarrow \textbf{DGVct$_{\mathrm{c}}$}$ is full and faithful since it is induced by the Dold-Kan equivalence. Moreover, if $V \in$ \textbf{DGVct$_{\mathrm{c}}$}, we may find $W \in$ \textbf{SVct$_{\mathrm{c}}$} so that $NW \cong V$. Since $\Gamma(V)_0 = V_0 = 0$, it follows that $\Gamma(V) \in$ \textbf{SVct$_{\mathrm{c}}$}. Setting $W = \Gamma(V)$ meets the required condition. Therefore, by \cite[Section 2.1, Proposition 3]{Par70}, we deduce that the restriction of the normalization functor $N$ induces an equivalence of categories between connected vector spaces with an inverse given by the restriction of $\Gamma$.
\end{proof}

The following result is dual to \cite[Part I, Proposition 4.5]{Qui69}.
\begin{lem} \label{lem:cof-isom}
Let $V$ be a differential graded vector space.
Then the following maps
$$H_\ast (\widetilde{N} T'_s \Gamma (V)) \longrightarrow H_\ast(T'_d (V)) \longrightarrow T'_d H_\ast (V)$$ 
of graded coalgebras are isomorphisms.
\end{lem}

\begin{proof}
These maps are obtained by using the universal properties of the respective tensor coalgebras.
Hence we deduce the isomorphisms by applying K\"unneth and Eilenberg-Zilber theorems.
\end{proof}

\begin{lem}
If $C$ is a cofree differential graded coalgebra, then the map $$\widetilde{N} R C \longrightarrow C$$ is a weak equivalence.
\end{lem}

\begin{proof}
Since $C \cong T'_d(V)$ one has
$$H_\ast (\widetilde{N} R C) \cong H_\ast (\widetilde{N}T'_s \Gamma (V)) \cong H_\ast (T'_d (V)) \cong H_\ast (C)$$
and hence the required result.
\end{proof}

\begin{thm} \label{mainresult}
If $C$ is a fibrant connected
differential graded coalgebra, then the map $$\widetilde{N} R C \longrightarrow C$$ is a weak equivalence. Hence, there is a Quillen equivalence between the category of connected differential graded coalgebras and the category of connected simplicial coalgebras.
\end{thm}

\begin{proof}
Recall in Lemma \ref{lem:retract} that $C$ is a retract of a cofree coalgebra $G$, which may be written as $T'_d(V)$ for some $V \in$ \textbf{DGVct}$_\mathrm{c}$. 
Applying the functors $R^{\mathrm{com}}$, $\widetilde{N}$ and $H_\ast$ to the retract map
$C \rightarrow T'_d(V) \rightarrow C$
 we find the diagram 
$$
\xymatrix{
H_\ast (\widetilde{N} R C) \ar[rr] \ar[d] && H_\ast (\widetilde{N}T'_s \Gamma (V)) \ar[rr] \ar[d]^\cong && H_\ast (\widetilde{N} R C) \ar[d] \\
H_\ast C \ar[rr]        && H_\ast (T'_d (V)) \ar[rr]        && H_\ast C
}$$
of homology morphisms. Since by Lemma \ref{lem:cof-isom}, $H_\ast (\widetilde{N} T'_s \Gamma (V)) \cong H_\ast(T'_d (V))$, we deduce that $H_\ast (\widetilde{N} R C) \cong H_\ast C$ with help of \cite[Lemma 2.7]{DS95}.
Then, the Quillen equivalence follows from \cite[Corollary 1.3.16]{Hov99}.
\end{proof}

\section{Appendix on connected differential graded algebras} \label{section:con-alg}

In this appendix, we consider the category of connected differential graded algebras. 
 Its main interest here is that a particular dual of its limits is used to construct colimits for the category \textbf{DGcoAlg}$_\mathrm{c}$.

\begin{defn}
A \textit{connected} differential graded algebra $A$ is a differential graded algebra with $A_0 =K$. 
We denote by \textbf{DGAlg}$_\mathrm{c}$ the category of connected differential graded algebras. 
\end{defn}

\begin{defn}
Let $A$ be an object in the category \textbf{DGAlg}$_\mathrm{c}$. The isomorphism  $\mu_{\mid_{A_0}} \colon K \rightarrow A_0$ induces a map $\gamma_A \colon A \rightarrow K[0]$ and we define a functor $I_d \colon \textbf{DGAlg}_\mathrm{c} \rightarrow \textbf{DGVct}_\mathrm{c}$ by $I_d (A) = \ker \gamma_A.$
\end{defn}

\begin{lem}
The tensor algebra functor $T_d \colon \textbf{DGVct}_\mathrm{c} \rightarrow \textbf{DGAlg}_\mathrm{c}$ is left adjoint to the functor $I_d$.
\end{lem}

\begin{prop}
The category of connected differential graded algebras is complete and cocomplete.
\end{prop}

\begin{proof}
Constructions of limits are well-known in the category of differential graded algebras. Because of connectedness, some refinements have to be performed. A terminal object in \textbf{DGAlg}$_\mathrm{c}$ is given by $K[0]$. If $f,g \colon A \rightarrow B$ are two maps in \textbf{DGAlg}$_\mathrm{c}$, their equalizer is given by $\ker \left(f-g\right)$. Now let $A$ and $B$ be objects in \textbf{DGAlg}$_\mathrm{c}$. We may form the maps $$A \times B \stackrel{\pi_A}{\longrightarrow} A \stackrel{\gamma_A}{\longrightarrow} K[0] \hspace{.2cm} \textrm{and} \hspace{.2cm} A \times B \stackrel{\pi_B}{\longrightarrow} B \stackrel{\gamma_B}{\longrightarrow} K[0].$$ 
Then the product of $A$ and $B$ in \textbf{DGAlg}$_\mathrm{c}$ is given by $$A \sqcap B = \ker \left(\gamma_A \circ \pi_A  - \gamma_B \circ \pi_B\right).$$ 
For colimits, we first notice that an initial object in \textbf{DGAlg}$_\mathrm{c}$ is given by $K[0]$. Then the coequalizer of the two maps $f, g \colon A \rightarrow B$ is given by $$B / \left\langle f(a)-g(a), a \in A\right\rangle$$ where $\left\langle f(a)-g(a), a \in A\right\rangle$ denotes the ideal generated by $f(a)-g(a)$ for $a \in A$. 
 In \cite{Jar97}, the coproduct of two differential graded non-commutative algebras $A$ and $B$ is given by factoring out from the tensor algebra $T_d\left(A \otimes B\right)$ the ideal $\mathcal{I}$ which is generated by elements of the form
\begin{eqnarray*}
(a_1 \otimes b_1)\otimes(1\otimes b_2)- a_1\otimes b_1b_2, \\
(a_1 \otimes 1)\otimes(a_2\otimes b_2)- a_1a_2\otimes b_2. 
\end{eqnarray*}
However, the resulting object need not to be connected even if $A$ and $B$ are connected. To avoid this problem, we define the coproduct of two objects in \textbf{DGAlg}$_\mathrm{c}$ as $$A \sqcup B = T_d\left(\ker \gamma_A \otimes \ker \gamma_B \right) / \mathcal{I}.$$
Since $\ker \gamma_A$ and $\ker \gamma_A $ are connected differential graded vector spaces, they will not contribute to the degree zero part of the tensor algebra $T_d\left(\ker \gamma_A \otimes \ker \gamma_B \right)$.
In this way, we will have $\left(A \sqcup B\right)_0 =K$ and therefore $A \sqcup B \in$ \textbf{DGAlg}$_\mathrm{c}$.
\end{proof}

\end{document}